\documentclass{amsart}

\usepackage[mathscr]{eucal}
\usepackage{amssymb}
\usepackage[bbgreekl]{mathbbol}
\usepackage[usenames,dvipsnames]{color}
\usepackage{amsthm}
\usepackage{bbold}
\usepackage{enumerate}
\usepackage{array}
\usepackage{amsmath}
\usepackage[bbgreekl]{mathbbol}

\usepackage{hyperref}

\hypersetup{colorlinks=true,citecolor=Brown,urlcolor=Brown,linkcolor=Brown}%

\numberwithin{equation}{section}
\usepackage[all]{xy}

\swapnumbers 

\newtheorem{Thm}[equation]{Theorem}
\newtheorem{Prop}[equation]{Proposition}

\newtheorem{Cor}[equation]{Corollary}

\theoremstyle{definition}
\newtheorem{Def}[equation]{Definition}
\newtheorem{Cons}[equation]{Construction}

\theoremstyle{remark}
\newtheorem{Rem}[equation]{Remark}
\newtheorem{Not}[equation]{Notation}
\newtheorem{Exa}[equation]{Example}

\newcommand{\nc}{\newcommand}
\nc{\dmo}{\DeclareMathOperator}

\dmo{\End}{End}
\dmo{\Hom}{Hom}
\dmo{\Ker}{Ker}
\dmo{\Pic}{Pic}
\dmo{\Res}{Res}
\dmo{\modname}{mod}
\dmo{\projname}{proj}
\dmo{\rmH}{H}
\dmo{\stabname}{stab}
\dmo{\torsname}{Tors}
\dmo{\gps}{gps}

\nc{\Cont}{\mathrm{Cont}} 
\nc{\Mid}{\,\big|\,}
\nc{\Quil}{\mathcal{S}_p}
\nc{\QG}{\Quil (G)}
\nc{\QP}{\Quil (P)}
\nc{\SET}[2]{\big\{\,#1\Mid#2\,\big\}}
\nc{\bbC}{\mathbb{C}}
\nc{\bbF}{\mathbb{F}}
\nc{\bbQ}{\mathbb{Q}}
\nc{\bbR}{\mathbb{R}}
\nc{\bbS}{\mathbb{S}}
\nc{\bbZ}{\mathbb{Z}}
\nc{\eg}{{\sl e.g.}}
\nc{\eps}{\epsilon}
\nc{\ie}{{\sl i.e.}\ }
\nc{\inv}{^{-1}}
\nc{\isoto}{\overset{\sim}{\,\to\,}}
\nc{\isotoo}{\overset{\sim}{\,\too\,}}
\nc{\ith}{^\textrm{th}}
\nc{\kk}{\Bbbk}
\nc{\mapstoo}{\longmapsto}
\nc{\onto}{\mathop{\twoheadrightarrow}}
\nc{\otoo}[1]{\overset{#1}{\,\too\,}}
\nc{\oto}[1]{\overset{#1}\to}
\nc{\potimes}[1]{^{\otimes #1}}
\nc{\qquadtext}[1]{\qquad\textrm{#1}\qquad}
\nc{\restr}[1]{_{|_{\scriptstyle #1}}}
\nc{\too}{\mathop{\longrightarrow}\limits}
\nc{\unit}{\mathbb{1}}
\nc{\Aname}{A}
\nc{\ACGP}{\Aname_{\bbC}(G,P)}
\nc{\mACGP}{\Tors{m}{\ACGP}}
\nc{\AGP}{\Aname_{\kk}(G,P)}
\nc{\Tk}{\mathrm{T}_{\kk}}
\nc{\TGP}{\Tk(G,P)}
\nc{\bbL}{\mathbb{L}}
\nc{\Tors}[2]{\torsname_{#1}{#2}}

\begin{document}


\title[Endotrivial representations and line bundles on~$\QG$]{endotrivial representations of finite groups and \\ equivariant line bundles on the Brown complex}
\author{Paul Balmer}
\date{\today}

\address{Paul Balmer, Mathematics Department, UCLA, Los Angeles, CA 90095-1555, USA}
\email{balmer@math.ucla.edu}
\urladdr{http://www.math.ucla.edu/$\sim$balmer}

\begin{abstract}
We relate endotrivial representations of a finite group in characteristic~$p$ to equivariant line bundles on the simplicial complex of non-trivial $p$-subgroups, by means of weak homomorphisms.
\end{abstract}


\thanks{Supported by NSF grant~DMS-1303073 and Research Award of the Humboldt Foundation.}

\maketitle

\begin{center}
\textit{Dedicated to Serge Bouc on the occasion of his 60$\ith$ birthday}
\end{center}

\bigbreak
\section{Introduction}

Let $G$ be a finite group, $p$ a prime dividing the order of~$G$ and $\kk$ a field of characteristic~$p$. For the whole paper, we fix a Sylow $p$-subgroup~$P$ of~$G$.

\medbreak

Consider the \emph{endotrivial} $\kk G$-modules $M$, \ie those finite dimensional $\kk$-linear representations $M$ of~$G$ which are $\otimes$-invertible in the stable category $\kk G\,\text{--}\stabname=\kk G\,\text{--}\modname/\,\kk G\,\text{--}\projname$; this means that the $\kk G$-module $\End_{\kk}(M)$ is isomorphic to the trivial module~$\kk$ plus projective summands. The stable isomorphism classes of these endotrivial modules form an abelian group, $\Tk(G)$, under tensor product. This important invariant has been fully described for $p$-groups in celebrated work of Carlson and Th\'evenaz~\cite{CarlsonThevenaz04,CarlsonThevenaz05}. Therefore, for general finite groups~$G$, the focus has moved towards studying the relative version:
\[
\TGP :=\Ker\big(\Tk(G)\to \Tk(P)\big)\,.
\]

We connect this piece of modular representation theory to the equivariant topology of the \emph{Brown complex $\QG$ of $p$-subgroups}, see~\cite{Brown75}. This \mbox{$G$-space~$\QG$} is the simplicial complex associated to the poset of nontrivial $p$-subgroups of~$G$, on which $G$ acts by conjugation. The study of~$\QG$ is a major topic in group theory, centered around Quillen's conjecture~\cite{Quillen78}, which predicts that if $\QG$ is contractible then it is $G$-contractible, \ie $G$ admits a non-trivial normal $p$-subgroup. Here, we focus on the Picard group~$\Pic^G(\QG)$ of $G$-equivariant complex line bundles on~$\QG$; see Segal~\cite{Segal68}.

Our main result, Theorem~\ref{thm:main}, relates those two theories as follows (see Cor.\,\ref{cor:main}):
\begin{Thm}
\label{thm:main-intro}%
Suppose $\kk$ algebraically closed. Then there exists an isomorphism
\[
\TGP \simeq \ \Tors{p'}{\Pic^G(\QG)}
\]
where $\Tors{p'}{\Pic^G(\QG)}$ is the prime-to-$p$ torsion subgroup of~$\Pic^G(\QG)$. 
\end{Thm}

The left-hand abelian group $\TGP$ is always finite; see Remark~\ref{rem:bar-k}. About the right-hand side, it is true for general finite $G$-CW-complex~$X$ that the group $\Pic^G(X)$ can be interpreted as an equivariant cohomology group, namely $\rmH^2_G(X,\bbZ)$; in particular it is a finitely generated abelien group; see Remark~\ref{rem:Burt}. Some readers will consider Theorem~\ref{thm:main-intro} as the topological answer to the modular-representation-theoretic problem of computing~$\TGP$.

Since its origin in~\cite{Brown75,Quillen78}, the space $\QG$ is related to the $p$-local study of~$G$. Closer to our specific subject, Kn\"orr and Robinson in~\cite{KnoerrRobinson89} and Th\'evenaz in~\cite{Thevenaz93} already exhibited interesting relations between modular representation theory and equivariant K-theory of~$\QG$. The connection we propose here does not only relate \emph{invariants} of both worlds but appears at a slightly deeper level, in that it connects actual objects. Indeed, in Construction~\ref{cons:L}, we build complex line bundles over~$\QG$ from endotrivial representations of~$G$. This construction then yields the isomorphism of Theorem~\ref{thm:main-intro}. It would actually be interesting to see whether similar constructions exist for other classes of modular representations of~$G$, beyond endotrivial ones. 

The attentive reader will appreciate that modular representations of~$G$ live in positive characteristic whereas complex line bundles on the space~$\QG$ are rather ``characteristic zero" objects. This cross-characteristic connection is made possible thanks to the use of torsion elements and roots of unity. More precisely, we use in a crucial way the re-interpretation~\cite{Balmer13b} of the group~$\TGP$ in terms of \emph{weak $P$-homomorphisms}. Let us remind the reader.
\begin{Def}
\label{def:weak-hom}%
Let $K$ be a field -- which will be either $\kk$ or~$\bbC$ in the sequel. A function $u:G\too K^*=K\!-\!\{0\}$ is a \emph{($K$-valued) weak $P$-homomorphism} if
\begin{enumerate}[(WH\,1)]
\item
\label{it:WH1}%
\quad $u(g)=1$ when $g\in P$.
\smallbreak
\item
\label{it:WH2}%
\quad $u(g)=1$ if $P\cap P^g=1$.
\smallbreak
\item
\label{it:WH3}%
\quad $u(g_2\,g_1)=u(g_2)\,u(g_1)$ if $P\cap P^{g_1}\cap P^{g_2g_1}\neq 1$.
\end{enumerate}
The name comes from (WH\,\ref{it:WH3}) which is a weakening of the usual homomorphism condition. We denote by~$\textrm{A}_K(G,P)$ the group of all weak $P$-homomorphisms from~$G$ to~$K^*$, equipped with pointwise multiplication: $(uv)(g)=u(g)\,v(g)$.
\end{Def}

The main result of~\cite{Balmer13b} is the existence of an explicit isomorphism
\begin{equation}
\label{eq:AT}%
\AGP \simeq \TGP\,.
\end{equation}
This result has already found interesting applications, for instance the computation of~$\TGP$ for new classes of groups by Carlson-Mazza-Nakano~\cite{CarlsonMazzaNakano14}~and Carlson-Th\'evenaz~\cite{CarlsonThevenaz15}. Here, we will use the complex version $\ACGP$ to build a homomorphism
\[
\bbL:\ACGP\to \Pic^G(\QG)
\]
which will yield the isomorphism of Theorem~\ref{thm:main-intro} when suitably restricted to torsion. Injectivity of~$\bbL$ on torsion relies in an essential way on a result of Symonds~\cite{Symonds98}, namely the contractibility of the orbit space~$\QG/G$.

As often in such matters, it is difficult to predict which way traffic will go on the new bridge opened by Theorem~\ref{thm:main-intro}. Computations of~$\TGP$ have already been performed for many classes of finite groups and it seems quite possible that these examples will produce new equivariant line bundles for people interested in the $G$-homotopy type of~$\QG$. Conversely, Theorem~\ref{thm:main-intro} might prove useful to modular representation theorists in endotrivial need. Only future work will tell.

Finally, we emphasize that the $G$-space $\QG$ can of course be replaced by any $G$-homotopically equivalent $G$-space, like Quillen's version~\cite{Quillen78} via elementary abelian $p$-subgroups, Bouc's variant~\cite{Bouc84}, or Robinson's, see Webb~\cite{Webb87}.

\section{The Brown complex and roots of functions}
\label{sec:setup}%

In this preparatory section, we gather some background and notation.

\begin{Not}
For an integer $m\geq 1$ and a field~$K$ (which will be $\kk$ or $\bbC$), we denote by~$\mu_m(K)=\SET{\zeta\in K}{\zeta^m=1}$ the group of $m\ith$ roots of unity in~$K$.
\end{Not}

\begin{Not}
\label{not:QG-Y}%
The Brown complex $\QG$ is the simplical complex with one non-degenerate $n$-simplex $[Q_0< Q_1 < \cdots < Q_n]$ for each sequence of $n$ proper inclusions of nontrivial $p$-subgroups, with the usual face-operations ``dropping~$Q_i$''. For \mbox{$n=0$}, we thus have a point $[Q]$ in~$\QG$ for each non-trivial $p$-subgroup~$Q\le G$. The space~$\QG$ admits an obvious \emph{right $G$-action} given by conjugation on the $p$-subgroups, that is $Q\cdot g:=Q^g=g\inv Q g$. This $G$-action on~$\QG$ is compatible with the cell structure.

Since we have fixed a Sylow $p$-subgroup $P\le G$, we can consider the subcomplex
\[
Y := \QP \,\subseteq\, \QG
\]
on those subgroups contained in~$P$, \ie we keep in~$Y$ those $n$-cells $[Q_0< \cdots < Q_n]$ of~$\QG$ corresponding to non-trivial subgroups of~$P$. This closed subspace $Y$ of~$\QG$ is contractible, for instance towards the point~$[P]$. But more than that, $Y$ is an $N$-subspace of~$\QG$ for $N=N_G(P)$ the normalizer of~$P$. As such, $Y$ is even $N$-contractible. See~\cite{ThevenazWebb91} if necessary. A fortiori, $Y$ is $P$-contractible. The translates $Yg=\Quil(P^g)$ of the closed subspace~$Y$ cover the space~$\QG$:
\[
 \QG = \cup_{g\in G} \ \Quil(P^g) = \cup_{g\in G} \ Yg.
\]
We shall perform several ``$G$-equivariant constructions'' over~$\QG$ by first performing a basic construction over~$Y$ and then showing that the translates of this basic construction on $Yg_1$ and on~$Yg_2$ agree on the intersection $Yg_1\cap Yg_2$ for all~$g_1,g_2$.
\end{Not}

\begin{Rem}
\label{rem:non-trivial}%
We will be tacitly using the following fact. For $g_1,\ldots, g_n\in G$ (typically with $n\le 3$), we have $P^{g_1}\cap \cdots \cap P^{g_n}\neq 1$ if and only if $Y{g_1}\cap \cdots \cap Y{g_n}$ is not empty. Clearly a nontrivial $P^{g_1}\cap \cdots \cap P^{g_n}$ gives a point in $Y{g_1}\cap \cdots \cap Y{g_n}$. Conversely, as $G$ acts simplicially on~$\QG$, a non-empty intersection $Y{g_1}\cap \cdots \cap Y{g_n}$ must contain some 0-simplex~$[Q]$, \ie some nontrivial $p$-subgroup $Q\le P^{g_i}$ for all~$i$.
\end{Rem}

We shall also often use the following standard notation:

\begin{Not}
\label{not:cdot-u}%
When $\lambda: L_1\to L_2$ is a map of complex line bundles on a space~$X$ and $\eps:X\to \bbC^*$ is a continuous function, we denote by~$\lambda\cdot\eps$ the map~$\lambda$ composed with the automorphism (of~$L_1$ or~$L_2$) which scales by~$\eps(x)$ the fiber over~$x$.
\end{Not}

\begin{Rem}
\label{rem:Pic}%
A $G$-equivariant complex line bundle $L$ over a (right) $G$-space~$X$ consists of a complex line bundle $\pi:L\to X$ such that $L$ is also equipped with a $G$-action making $\pi$ equivariant and such that the action of every~$g\in G$ on fibers $L_x\to L_{xg}$ is $\bbC$-linear. More generally, see~\cite{Segal68} for $G$-equivariant vector bundles. We denote by $\Pic^G(X)$ the group of $G$-equivariant isomorphism classes of such~$L$, equipped with tensor product. The contravariant functor $\Pic^G(-)$ is invariant under $G$-homotopy. In particular, if $X$ is $G$-equivariantly contractible, the map $\Hom_{\gps}(G,\bbC^*)\cong\Pic^G(\ast)\too\Pic^G(X)$ is an isomorphism.

In the case of $X=\QG$, restriction to the $P$-subspace~$Y=\Quil(P)$ yields a group homomorphism from $\Pic^G(\QG)$ to the one-dimensional complex representations of~$P$, that we shall simply denote by~$\Res^G_P$
\begin{equation}
\label{eq:Res-Pic}%
\Res^G_P:\Pic^G(\QG)\to \Pic^P(\QP)\cong\Hom_{\gps}(P,\bbC^*).
\end{equation}
\end{Rem}

\begin{Rem}[Totaro]
\label{rem:Burt}%
For a compact Lie group~$G$ acting on a manifold~$M$, there is an isomorphism~$\Pic^G(M)\simeq \rmH^2_G(M,\bbZ)=\rmH^2(M\times_G E G,\bbZ)$, where $E G\to B G$ is the universal $G$-principal bundle on the classifying space~$BG$; see~\cite[Thm.\,C.47]{GuilleminGinzburgKarshon02}, where the similar result for a finite group acting on a finite CW-complex is attributed to~\cite{HattoriYoshida76}. Alternatively, one can see the latter by reducing to the case of manifolds, since every finite $G$-CW-complex is $G$-homotopy equivalent to a (noncompact) $G$-manifold. Then the group~$\rmH^2(X\times_G EG,\bbZ)$ can be approached via a Serre spectral sequence for the fibration $X\to X\times_G EG \to BG$. In particular, using that $G$ is finite, the spectral sequence collapses rationally to an isomorphism $\rmH^2(X\times_G EG,\bbQ)\simeq \rmH^0(BG,\rmH^2(X,\bbQ))$ showing that $\Pic^G(X)\otimes\bbQ \simeq (\Pic(X)\otimes\bbQ)^G$.
\end{Rem}

\begin{Not}
\label{not:g_*}%
For a subspace $Y$ of a $G$-space~$X$, like our $Y=\QP\subseteq\QG=X$, every element~$g\in G$ yields a homeomorphism $\cdot g: Y\isoto Yg$. We can transport things from~$Y$ to~$Yg$ via this homeomorphism, and we use $g_*(-)$ to denote this idea. For instance, if $f:Y\to \bbC$ is a function, then $g_*f:Yg\to \bbC$ is $g_*f(x):=f(xg\inv)$. Another situation will be that of $G$-equivariant line bundles~$L\oto{\pi} X$ and $L'\oto{\pi'} X$ and a morphism $\lambda:L\restr{Y}\to L'\restr{Y}$ of bundles over~$Y$, in which case the morphism $g_*\lambda:L\restr{Yg}\to L'\restr{Yg}$ is defined by the commutativity of the following top face:
\begin{equation}
\label{eq:g_*(lambda)}%
\vcenter{\xymatrix@C=2em@R=1.5em{
L\restr{Y} \ar[rd]^-{\lambda} \ar[rdd]_-{\pi}
 \ar[rrr]^-{\cdot g}_-{\simeq}
&&& L\restr{Yg} \ar[rd]^(.6){=:\displaystyle g_*(\lambda)} \ar[rdd]_(.7){\pi}|-{\hole}
\\
& L'\restr{Y} \ar[rrr]^-{\cdot g}_-{\simeq} \ar[d]^-{\pi'}
&&& L'\restr{Yg} \ar[d]^-{\pi'}
\\
& Y \ar[rrr]^-{\cdot g}_-{\simeq}
&&& Yg\,.\kern-.5em
}}
\end{equation}
As we use \emph{right} actions (that is $(\cdot g_2g_1)=(\cdot g_1)\circ(\cdot g_2)$) we have $(g_2g_1)_*=(g_1)_*\circ (g_2)_*$.
\end{Not}

Let us now say a word of roots of complex functions.
\begin{Rem}
\label{rem:C-triv}%
Throughout the paper,~$\bbC$ is given the trivial $G$-action. Hence a $G$-map $f:X\to \bbC$ from a (right) $G$-space~$X$ to~$\bbC$ is simply a continuous function such that $f(xg)=f(x)$ for all~$x\in X$ and all~$g\in G$, that is essentially a continuous function $\bar f:X/G\to \bbC$ on the orbit space.
\end{Rem}

\begin{Prop}
\label{prop:root}%
Let $m\geq1$ be an integer, $X$ a $G$-space and $f:X\to \bbC^*$ a $G$-map. Suppose that $f$ is $G$-homotopic to the constant map~$1$. Then $f$ admits an $m\ith$ root in~$\Cont_G(X,\bbC^*)$, \ie a $G$-map $f^{1/m}:X\to \bbC^*$ such that $(f^{1/m})^m=f$.
\end{Prop}

\begin{proof}
By assumption, the induced map $\bar f:X/G\to \bbC^*$ is homotopic to~$1$. Then it suffices to observe that $\bar f$ has an $m\ith$ root by a standard determination-of-the-logarithm argument. (Let $\bar X=X/G$ and let $H:\bar X\times[0,1]\to \bbC^*$ be a homotopy between $H(x,0)=\bar f(x)$ and $H(x,1)=1$. Lifting each $t\mapsto H(x,t)/|H(x,t)|\in \bbS^1$ along the fibration $\bbR\onto \bbS^1$, we find a map $\theta:\bar X\times [0,1]\to \bbR$ such that $H(x,t)=|H(x,t)|\cdot e^{i\theta(x,t)}$ and $\theta(x,1)=0$. One can then define the $m\ith$ root of~$\bar f$ via $\bar f^{1/m}(x)=|\bar f(x)|^{1/m}\cdot e^{i\theta(x,0)/m}$ for all $x\in\bar X$.)
\end{proof}

\begin{Cor}
\label{cor:root}%
If $X/G$ is contractible (\eg\ if~$X$ is $G$-contractible) then for every integer~$m\geq 1$, every $G$-map $f:X\to \bbC^*$ admits an $m\ith$ root $f^{1/m}\in\Cont_G(X,\bbC^*)$.
\end{Cor}

\begin{proof}
As such a map $f$ factors via $X\onto X/G$, the contractibility of~$X/G$ implies that $f$ is $G$-homotopically trivial and we conclude by Proposition~\ref{prop:root}.
\end{proof}

\begin{Cor}
\label{cor:Symonds}%
For every integer~$m\geq 1$, every $G$-map $f:\QG\to \bbC^*$ on the Brown complex admits an $m\ith$ root $f^{1/m}\in\Cont_G(\QG,\bbC^*)$.
\end{Cor}

\begin{proof}
The orbit space $\QG/G$ is contractible by Symonds~\cite{Symonds98}.
\end{proof}

\section{Constructing line bundles from weak homomorphisms}
\label{sec:u-to-L}%

We now want to associate a $G$-equivariant complex line bundle $L_u$ on~$\QG$ to each complex-valued weak homomorphism $u\in \ACGP$ as in Definition~\ref{def:weak-hom}. In essence, this is a very standard gluing procedure, familiar to every geometer. We spell out some details for the sake of clarity and to see where the ``weak homomorphism'' conditions~(WH\,\ref{it:WH1}-\ref{it:WH3}) show up.

\begin{Cons}
\label{cons:L}%
Let $u:G\to \bbC^*$ be a weak $P$-homomorphism and $Y=\QP\subseteq\QG$ as in Notation~\ref{not:QG-Y}. Define $L_u$ as the following topological space:
{\setlength{\fboxsep}{1em}
\[
\boxed{L_u:=\Big(\bigsqcup_{s\in G}\ Ys\times \bbC\ \Big)_{\displaystyle\big/\sim}}
\]
}%
where $\sim$ is the equivalence relation defined in~\eqref{eq:sim} below. We use the notation $(y,a)_s$ to indicate a point $(y,a)$ in the space~$Ys\times \bbC$ with index~$s\in G$; and we shall write $[y,a]_s\in L_u$ for its class modulo~$\sim$. (As the subsets $Ys$ do intersect in~$\QG$, the lighter notation $(y,a)$ would be ambiguous.) Note that the weak $P$-homomorphism~$u$ does not appear so far; it is used in the equivalence relation:
\begin{equation}
\label{eq:sim}%
 (y,a)_s\sim (z,b)_t
 \qquadtext{iff}
 \left\{
 \begin{array}{l}
  y=z
  \\
  \ \textrm{and}
  \\
  a\cdot u(s\,t\inv)=b.
 \end{array}
 \right.
\end{equation}
One easily verifies that~$\sim$ is indeed an equivalence relation, using~(WH\,\ref{it:WH3}). Of course, $L_u$ is equipped with the quotient topology.
\end{Cons}

\begin{Rem}
\label{rem:mnemo}%
A good way to keep track of what happens is to think of the class $[y,a]_s$ as a fictional element ``$a\cdot s\in \bbC$ living in a fiber over~$y\in\QG$", which is not defined since we do not know how $s\in G$ should act on~$\bbC$. Still, equality between ``$a\cdot s$ over~$y$" and ``$b\cdot t$ over~$z$" should nonetheless mean that they live in the same fiber, \ie $y=z$, and that ``$a\cdot (st\inv)=b$''. So we decide that the action of $st\inv$, \ie the \emph{difference} of the two actions over the point $y=z$ in~$Ys\cap Yt$, is given via the weak homomorphism~$u$. This can be compared to~\cite[Eq.\,(2.7)]{Balmer13b}.
\end{Rem}

The space~$L_u$ admits a continuous projection to the Brown complex
\[
\pi_u:L_u \to \QG
\]
simply given by $[y,a]_s \mapsto y$ and whose fibers are isomorphic to~$\bbC$. More precisely, for every $s\in G$, we have a homeomorphism
\begin{equation}
\label{eq:alpha}%
\xymatrix@R=.1em{
\alpha_s:
&&
\kern-4em \unit_{Ys}:=Ys \times \bbC \ar[r]^-{\simeq}
& \pi_u\inv(Ys)\subseteq L_u\kern-2em
\\
&& (y,a) \ar@{|->}[r]
& [y,a]_s
}
\end{equation}
(We denote trivial line bundles by~$\unit$.) These are \emph{trivializations} of~$L_u$ over~$Ys$. For all $s,t\in G$, the transition maps~$\alpha_t\inv\alpha_s$ on the intersection
\[
\xymatrix@R=.2em{
(Ys\cap Yt) \times \bbC \ar[r]^-{\alpha_s}_-{\simeq}
& \pi_u\inv(Ys\cap Yt)
& (Ys\cap Yt) \times \bbC \ar[l]_-{\alpha_t}^-{\simeq}
\\
(y,a) \ar@{|->}[r]
& [y,a]_s \overset{\textrm{\eqref{eq:sim}}}= [y,a\cdot u(s\,t\inv)]_t\ar@{|->}[r]
& (y,a\cdot u(s\,t\inv))}
\]
is given by the (constant) linear isomorphism, multiplication by the unit~$u(s\,t\inv)$. In other words, $L_u\otoo{\pi_u} \QG$ is a complex line bundle on~$\QG$. We record the above computation in compact form\,: for all $s,t\in G$ we have an equality
\begin{equation}
\label{eq:trans-alpha}%
\alpha_s=\alpha_t\cdot u(s\,t\inv)\qquad\textrm{over }Ys\cap Yt
\end{equation}
as isomorphisms $\unit_{Ys\cap Yt}\isoto (L_u)\restr{Ys\cap Yt}$. Here we used Notation~\ref{not:cdot-u}.

\smallbreak

The right $G$-action on the space~$L_u$ is defined, in the spirit of Remark~\ref{rem:mnemo}, by
\[
[y,a]_s\cdot g:=[yg,a]_{sg}.
\]
This action clearly makes~$\pi_u:L_u\to \QG$ into a $G$-map. In view of the above, $G$ acts linearly on the fibers of~$\pi_u$ and thus makes~$L_u$ into a $G$-equivariant complex line bundle over~$\QG$. We can also observe that the collection of local trivializations $\alpha_s:\unit_{Ys} \isotoo (L_u)\restr{Ys}$ given in~\eqref{eq:alpha} is ``\emph{$G$-coherent}" (\footnote{\,We do not say ``$G$-equivariant'' to avoid confusion.}) by which we mean that for all $s,g\in G$ we have
\begin{equation}
\label{eq:g_*(alpha)}%
g_*(\alpha_s)=\alpha_{sg}
\end{equation}
as isomorphisms $\unit_{Ysg}\isoto (L_u)\restr{Ysg}$. This fact results directly from the definitions, see~\eqref{eq:g_*(lambda)} and~\eqref{eq:alpha}. Combining this with~\eqref{eq:trans-alpha} we note for later use the formula:
\begin{equation}
\label{eq:g_*(alpha_1)}%
g_*(\alpha_1)=\alpha_1\cdot u(g)\qquad\textrm{over }Y\cap Yg
\end{equation}
as isomorphisms $\unit_{Y\cap Yg}\isoto (L_u)\restr{Y\cap Yg}$, for all $g\in G$ such that $P\cap P^g\neq1$.

\begin{Prop}
\label{prop:L_uv}%
For any two weak $P$-homomorphisms $u,v\in \ACGP$ we have a $G$-equivariant isomorphism $L_{uv}\simeq L_u\otimes L_v$ of complex line bundles over~$\QG$.
\end{Prop}

\begin{proof}
Note that the trivializations~\eqref{eq:alpha} of~$L_u$ are performed on the closed cover of~$\QG$ given by~$(Ys)_{s\in G}$, which is independent of~$u$. So, it is the same cover for~$L_u$, $L_v$ and $L_{uv}$. The statement then follows from the observation that the following obvious isomorphisms over~$Ys$ (where we temporarily decorate the three morphisms~$\alpha$ as $\alpha^{(u)}$, $\alpha^{(v)}$ and $\alpha^{(uv)}$ to distinguish the respective line bundles)
\[
\xymatrix{
(L_u\otimes L_v)\restr{Ys}
\cong (L_u)\restr{Ys}\otimes (L_v)\restr{Ys}
&& \unit_{Ys}\otimes\unit_{Ys} \cong \unit_{Ys}
 \ar[ll]^-{\simeq}_-{\alpha_s^{(u)}\otimes\alpha_s^{(v)}} \ar[r]^-{\alpha_s^{(uv)}}
& (L_{uv})\restr{Ys}
}
\]
patch together into a $G$-equivariant isomorphism $L_u\otimes L_v\isoto L_{uv}$ on~$\QG$. Verification of this patching is immediate from~\eqref{eq:trans-alpha} and the following agreement:
\[
\xymatrix{
\unit_{Ys\cap Yt} \otimes \unit_{Ys\cap Yt}
\cong \unit_{Ys\cap Yt}
 \ar@<-2.3em>[d]|-{(\cdot u(st\inv))\otimes (\cdot v(st\inv))} \ar@<3em>[d]^-{\cdot uv(st\inv)}
\\
\unit_{Ys\cap Yt} \otimes \unit_{Ys\cap Yt}
\cong \unit_{Ys\cap Yt}
}
\]
on the trivial bundle. Finally, the map $L_u\otimes L_v\isoto L_{uv}$ is $G$-equivariant because each $\{\alpha_s^{(...)}\}_{s\in G}$ is a $G$-coherent collection of maps, as we saw in~\eqref{eq:g_*(alpha)}.
\end{proof}

The reader will easily verify the following naturality of Construction~\ref{cons:L}.

\begin{Prop}
\label{prop:nat}%
Let $G'\le G$ be a subgroup containing~$P$ and consider the $G'$-subspace $\Quil(G')\subseteq \Quil(G)$. Then the following diagram
\[
\xymatrix{
\ACGP \ar[r]^-{\bbL} \ar[d]_-{\Res}
& \Pic^G(\QG) \ar[d]^-{\Res}
\\
\Aname_{\bbC}(G',P) \ar[r]^-{\bbL}
& \Pic^{G'}(\Quil(G'))
}
\]
is commutative.
\qed
\end{Prop}

\begin{Exa}
\label{exa:trivial}%
Let $u:G\to \bbC^*$ be a group homomorphism, \ie a one-dimensional representation. Assume that $u$ is trivial on~$P$. One associates to $u$ a weak $P$-homomorphism $\tilde u\in\ACGP$ by forcing~(WH\,\ref{it:WH2}), \ie by setting for every $g\in G$
\begin{equation}
\label{eq:tilde}%
\tilde u(g):=
\left\{
\begin{array}{cl}
u(g) & \textrm{ if }P\cap P^g\neq 1
\\
1 & \textrm{ if }P\cap P^g = 1.
\end{array}
\right.
\end{equation}
Then $L_{\tilde u}$ is isomorphic to the ``constant'' line bundle (in the sense of~\cite{Segal68}), that is, the line bundle $\unit_{u}:=\QG\times \bbC$ with action $(y,a)\cdot g=(yg,au(g))$. Indeed, inspired by Remark~\ref{rem:mnemo}, one easily guesses the $G$-equivariant isomorphism $L_{\tilde u}\isoto \unit_{u}$ by sending the class $[y,a]_s$ in $L_{\tilde u}$ (see Construction~\ref{cons:L}) to the point $(y,a\cdot u(s))$ in~$\QG\times\bbC=\unit_{u}$. Verifications are left to the reader.
\end{Exa}

The modification~\eqref{eq:tilde} of~$u$ into a weak homomorphism~$\tilde u$ is irrelevant for the construction of~$L_{\tilde u}$ since~\eqref{eq:sim} only uses values $\tilde u(g)$ over the subset~$Y\cap Yg$. Indeed, either $P\cap P^g=1$ and this subset is empty, or $P\cap P^g\neq1$ and $\tilde u(g)=u(g)$ anyway. Furthermore, the homomorphism $u\mapsto \tilde u$ is often injective, even after (post-) composition with~$\bbL$. We do not use the latter but state it for peace of mind:

\begin{Prop}
Suppose that $\QG$ is connected. Let $u:G\to \bbC^*$ be a group homomorphism which is trivial on~$P$ and such that the $G$-equivariant line bundle $\unit_u\simeq\bbL(\tilde u)$ is $G$-equivariantly trivial on~$\QG$ (for instance if $\tilde u=1$). Then $u=1$.
\end{Prop}

\begin{proof}
A $G$-equivariant isomorphism $\unit\isoto \unit_u$ is given by multiplication by a map $f:\QG\to \bbC^*$ such that $f(xg)=f(x)\cdot u(g)$ for all~$g\in G$ and $x\in \QG$. Choose an integer $m\geq 1$ such that $u(g)^m=1$. Then $f^m:\QG\to \bbC^*$ is a $G$-map. By Corollary~\ref{cor:Symonds}, this $f^m$ admits an $m\ith$ root in~$\Cont_G(\QG,\bbC^*)$, \ie there exists a $G$-map $\hat f:\QG\to \bbC^*$ such that $\hat f^m=f^m$. Since~$\QG$ is assumed connected, we have $\hat f=f\cdot \rho$ for some constant $\rho\in \mu_m(\bbC)$. Then $f$ is also a $G$-map and the above relation $f(xg)=f(x)\cdot u(g)$ forces $u(g)=1$ for all~$g\in G$.
\end{proof}

Assuming~$\QG$ connected is a mild condition. According to~\cite[Prop.\,5.2]{Quillen78}, if~$\QG$ is disconnected then the stabilizer $H\le G$ of a component is a strongly $p$-embedded subgroup, and our discussion can be safely reduced from~$G$ to~$H$.

\section{The results}
\label{sec:results}%

We now prove our main result, from which we will deduce Theorem~\ref{thm:main-intro} stated in the Introduction.  We saw in Proposition~\ref{prop:L_uv} that the assignment $u\mapsto L_u$ of Construction~\ref{cons:L} induces a well-defined homomorphism $\bbL:\ACGP\to \Pic^G(\QG)$ from the group of complex-valued weak $P$-homomorphisms (Def.\,\ref{def:weak-hom}) to the $G$-equivariant Picard group (Rem.\,\ref{rem:Pic}) of the Brown complex~$\QG$.

\begin{Thm}
\label{thm:main}%
The homomorphism $\bbL:\ACGP\to \Pic^G(\QG)$ is injective on torsion subgroups (denoted $\torsname$) and its image is detected by restriction to one-dimensional representations of~$P$, see~\eqref{eq:Res-Pic}. That is, we have an exact sequence
\begin{equation}
\label{eq:main}%
\xymatrix{
0 \ar[r]
& \Tors{}{\ACGP} \ar[r]^-{\bbL}
& \Tors{}{\Pic^G(\QG)} \ar[r]^-{\Res^G_P}
& \Hom_{\gps}(P,\bbC^*)\,.}
\end{equation}
Consequently, for every integer $m\geq 1$ prime to~$p$, our~$\bbL$ restricts to an isomorphism on the $m$-torsion subgroups\,(\,\footnote{\,By ``$m$-torsion" we mean exactly the annihilator of~$m$ itself, not of powers of~$m$.})
\[
\bbL:\mACGP \isoto \Tors{m}{\Pic^G(\QG)}.
\]
\end{Thm}

\begin{proof}
The proof will occupy the next couple of pages. First note that by naturality of~$\bbL$ (Prop.\,\ref{prop:nat} applied to $G'=P$), the following square commutes:
\[
\xymatrix{
\ACGP \ar[r]^-{\bbL} \ar[d]_-{\Res}
& \Pic^G(\QG) \ar[d]^-{\Res}
\\
\kern-2em 0= \Aname_{\bbC}(P,P) \ar[r]^-{\bbL}
& \Pic^{P}(\Quil(P)) \cong \Hom_{\gps}(P,\bbC^*)\,.\kern-5em
}
\]
This proves that $\Res^G_P\circ \bbL$ is trivial (even outside torsion).

We now prove injectivity of~$\bbL$ on the torsion of~$\ACGP$. Let $u\in\ACGP$ be an element of $m$-torsion for some~$m\geq 1$, meaning that $u(g)^m=1$ for all~\mbox{$g\in G$}. Suppose that we have a $G$-equivariant trivialization $\psi:\unit_{\QG}\isoto L_u$ of the line bundle $\bbL(u)=L_u$ (see Constr.\,\ref{cons:L}). Comparing the restriction $\psi\restr{Y}$ to the trivialization~$\alpha_1:\unit_Y\isoto (L_u)\restr{Y}$ given in~\eqref{eq:alpha}, we find a $P$-map $\delta:Y\to \bbC^*$ with
\[
\psi\restr{Y}=\alpha_1 \cdot \delta
\]
as isomorphisms $\unit_Y\isoto (L_u)\restr{Y}$. Combining the $G$-equivariance of~$\psi$ with the relation $g_*(\alpha_1)=\alpha_1\cdot u(g)$ on~$Y\cap Yg$ from~\eqref{eq:g_*(alpha_1)}, we see that for every $g\in G$ such that $P\cap P^g\neq 1$, we have for every~$y\in Y\cap Yg$
\begin{equation}
\label{eq:u-delta}%
u(g)=\frac{\delta(y)}{g_*\delta(y)}=\frac{\delta(y)}{\delta(y g\inv)}.
\end{equation}
As the left-hand side belongs to~$\mu_m(\bbC)$, we deduce that $\delta^m$ and $g_*(\delta^m)$ agree on the intersection~$Y\cap Yg$. Consequently the family of functions ~$(g_*(\delta^m))_{g\in G}$ patch together into a $G$-map $f:\QG\to\bbC^*$ by setting $f(x)=\delta(xg\inv)^m$ whenever~$x\in Yg$. By Corollary~\ref{cor:Symonds}, $f$ admits an $m\ith$ root, \ie there exists a $G$-map $f^{1/m}:\QG\to \bbC^*$ such that $(f^{1/m})^m=f$. On~$Y$, the two roots $f^{1/m}$ and~$\delta$ of the same map~$f$ must differ by an $m\ith$ root $\rho\in\mu_m(\bbC)$ which must be constant since~$Y$ is connected, say $\delta=\rho\cdot f$. But then for every $g\in G$ such that $P\cap P^g\neq1$ and for any $y\in Y\cap Yg\neq\varnothing$ (for which $yg\inv\in Y$ too), relation~\eqref{eq:u-delta} becomes
\[
u(g)=\frac{\delta(y)}{\delta(y g\inv)}=\frac{\rho\cdot f(y)}{\rho\cdot f(y g\inv)}=1
\]
by $G$-equivariance of~$f$. In the other case where $P\cap P^g=1$, we have $u(g)=1$ by~(WH\,\ref{it:WH2}). In short, $u=1$ is trivial. This proof uses the contractibility of~$\QG/G$, since Corollary~\ref{cor:Symonds} relies on Symonds~\cite{Symonds98}.

\medbreak

We now prove exactness of~\eqref{eq:main} in the middle via another construction.

\begin{Cons}
\label{cons:u}%
Let~$L$ be a $G$-equivariant complex line bundle on~$\QG$, which is torsion and such that $\Res^G_P(L)=1$, \ie $L$ restricts to the trivial $P$-bundle on~$\QP$. \emph{Choose} for some $m\geq 1$ a $G$-equivariant isomorphism
\[
\omega:\unit_{\QG}\isotoo L\potimes{m}
\]
over~$\QG$ and \emph{choose} a $P$-equivariant isomorphism over~$Y=\QP$
\[
\beta:\unit_Y \ \isotoo \ L\restr{Y}
\]
between the trivial bundle $\unit_Y=Y\times \bbC$ and the restriction of~$L$ to~$Y$. The $P$-equivariance of~$\beta$ means that, for every $h\in P$, we have
\begin{equation}
\label{eq:beta-P-equiv}%
h_*(\beta)=\beta
\end{equation}
as isomorphisms $\unit_Y\isoto L\restr{Y}$. There is a choice in the isomorphism~$\beta$, and we can replace $\beta$ by $\beta\cdot\delta$ for any $P$-map $\delta:Y\to \bbC^*$. We shall use this flexibility shortly.

Observe that $\beta\potimes{m}$ yields another trivialization of~$L\potimes{m}$ on~$Y$, that we can compare to the initial~$\omega$, restricted to~$Y$. It follows that we have $\omega\restr{Y}=\beta\potimes{m}\cdot\eps$ for some $P$-map~$\eps:Y\to \bbC^*$. Since the space $Y$ is $P$-contractible, Corollary~\ref{cor:root} produces an $m$-root of~$\eps$, say $\eps^{1/m}\in\Cont_P(Y,\bbC^*)$ with $(\eps^{1/m})^m=\eps$. Using this unit to replace $\beta$ by~$\beta \cdot \eps^{1/m}$, we can and shall assume that~$\beta:\unit_Y\isoto L\restr{Y}$ moreover satisfies
\begin{equation}
\label{eq:beta-omega}%
\beta\potimes{m}=\omega\restr{Y}.
\end{equation}

Then, for each $g\in G$, consider as before the translate $Yg=\Quil(P^g)\subseteq \QG$ and transport~$\beta$ into an isomorphism $g_*(\beta):\unit_{Yg}\isoto L\restr{Yg}$; see~\eqref{eq:g_*(lambda)}. If the isomorphisms $\beta$ and $g_*(\beta)$ were to agree on the intersection of their domains of definition~$Y\cap Yg$ for all~$g\in G$, the collection of isomorphisms $(g_*(\beta))_{g\in G}$ would patch together into a global isomorphism $\unit_{\QG}\isoto L$, automatically $G$-equivariant by construction. Since this cannot happen for nontrivial~$L$, there is an obstruction, and this happens to be a weak $P$-homomorphism. Indeed, for every $g\in G$ such that $P\cap P^g\neq1$, define what is \textsl{a priori} a function $u_L(g)\in\Cont(Y\cap Yg,\bbC^*)$ by 
\begin{equation}
\label{eq:beta-cocycle}%
g_*(\beta) = \beta \cdot u_L(g)\qquad\textrm{ over }Y\cap Yg
\end{equation}
\ie by the commutativity of the following diagram of line bundles on~$Y\cap Yg$:
\begin{equation}
\label{eq:u-def}%
\vcenter{\xymatrix@C=6em@R=1.5em{
\unit_{Y\cap Yg}
 \ar[r]^-{(g_*(\beta))\restr{Y\cap Yg}}_-{\simeq}
 \ar[d]_-{\cdot u_L(g):=}^-{\simeq}
& (L\restr{Yg})\restr{Y\cap Yg}=L\restr{Y\cap Yg}
 \ar@{=}@<3em>[d]
\\
\unit_{Y\cap Yg}
 \ar[r]^-{\beta\restr{Y\cap Yg}}_-{\simeq}
& (L\restr{Y})\restr{Y\cap Yg}=L\restr{Y\cap Yg}\,.\kern-.5em
}}
\end{equation}
There is no choice at this step. By convention, we set
\begin{equation}
\label{eq:conv}%
u_L(g)=1 \qquadtext{if} P\cap P^g=1.
\end{equation}%
In the case $P\cap P^g\neq 1$, we are going to prove that $u_L(g):Y\cap Yg\to \bbC^*$ is a constant function. Taking~\eqref{eq:u-def} to the $m\ith$ tensor power, replacing both instances of~$\beta\potimes{m}$ by~$\omega$ thanks to~\eqref{eq:beta-omega} and using that $\omega$ is $G$-equivariant, we deduce that $(u_L(g))^m=1$ on~$Y\cap Yg$. Since this space is non-empty and connected (even contractible), this implies that the function~$u_L(g)$ is actually constant, with value equal some complex $m\ith$ root of unity~$u_L(g)\in \mu_m(\bbC)$. In other words, the function
\[
u_L:G\to \mu_m(\bbC)\ ,\qquad g\mapsto u_L(g)
\]
is a candidate to be a complex-valued weak $P$-homomorphism. It satisfies~(WH\,\ref{it:WH1}) by $P$-equivariance of~$\beta$, see~\eqref{eq:beta-P-equiv} and~\eqref{eq:u-def} for $g=h\in P$; and $u_L$ satisfies~(WH\,\ref{it:WH2}) by definition~\eqref{eq:conv}. To verify the last property~(WH\,\ref{it:WH3}), consider $g_1,g_2\in G$ such that $P\cap P^{g_1}\cap P^{g_2g_1}\neq 1$, \ie such that the subset $Z:=Y\cap Yg_1\cap Yg_2g_1$ is non-empty. Then juxtaposing the defining diagram~\eqref{eq:u-def} for~$u_L(g_1)$ and the one for $u_L(g_2)$ transported by~$(g_1)_*$, both suitably restricted to this triple intersection~$Z$, we obtain the following commutative diagram over~$Z$:
\begin{equation}
\vcenter{\xymatrix@C=10em@R=1.5em{
\unit_{Z}
 \ar[r]^-{({g_{1}}_*{g_{2}}_*(\beta))\restr{Z}}_-{\simeq}
 \ar[d]_-{{g_{1}}_*(\cdot u_L(g_2))=\cdot u_L(g_2)}^-{\simeq}
& L\restr{Z}
 \ar@{=}[d]
\\
\unit_{Z}
 \ar[r]^-{({g_{1}}_*\beta)\restr{Z}}_-{\simeq}
 \ar[d]_-{\cdot u_L(g_1)}^-{\simeq}
& L\restr{Z}
 \ar@{=}[d]
\\
\unit_{Z}
 \ar[r]^-{\beta\restr{Z}}_-{\simeq}
& L\restr{Z}\,.\kern-.5em
}}
\end{equation}
We used at the top left that ${g_1}_*(-)$ is $\bbC$-linear. Using now that ${g_1}_*{g_2}_*=(g_2g_1)_*$, the left-hand vertical composite satisfies the commutativity expected of~$u_L(g_2g_1)$, \ie fits in place of~$u_L(g_2g_1)$ in~\eqref{eq:u-def} for $g=g_2g_1$, after restriction of the latter to~$Z$. This is where we use that $Z\neq\varnothing$ to deduce that $u_L(g_2g_1)=u_L(g_2)\cdot u_L(g_1)$.

It is interesting to see the parallel of these arguments with those of~\cite{Balmer13b}, where the non-emptiness of $Z$ is replaced by the non-vanishing of a suitable stable category. Both properties are equivalent, namely they both are avatars of the fact that the Sylow~$P$ and its conjugates $P^{g_1}$ and $P^{g_2g_1}$ intersect non-trivially.
\end{Cons}

At this stage, we have associated a weak $P$-homomorphism $u_L\in\mACGP$ to an $m$-torsion $G$-equivariant line bundle~$L$ on~$\QG$ and choices of isomorphisms $\omega:\unit_{\QG}\isoto L\potimes{m}$ and $\beta:\unit_Y\isoto L\restr{Y}$ satisfying~\eqref{eq:beta-omega}. We now claim that \mbox{$\bbL(u_L)\simeq L$}. For this, recall the line bundle~$L_{u_L}$ of Construction~\ref{cons:L}, which describes~$\bbL(u_L)$. It comes with an isomorphism $\alpha_1:\unit_{Y}\isoto (L_{u_L})\restr{Y}$ satisfying
\[
g_*(\alpha_1)=\alpha_1\cdot u_L(g)\qquad\textrm{ over }Y\cap Yg
\]
by~\eqref{eq:g_*(alpha_1)}. Comparing this formula to the similar one for~$\beta$ in~\eqref{eq:beta-cocycle}, we see that the following isomorphism $\varphi:=\beta\circ\alpha_1\inv$ over~$Y$
\[
\xymatrix{
\varphi\,:\ (L_{u_L})\restr{Y} \ar[r]_-{\simeq}^-{\alpha_1\inv}
& \unit_{Y} \ar[r]_-{\simeq}^-{\beta}
& L\restr{Y}
}\]
satisfies $g_*(\varphi)=\varphi$ on~$Y\cap Yg$ for all~$g\in G$. Therefore, the $(g_*\varphi)_{g\in G}$ patch together into a morphism $\varphi:L_{u_L}\to L$ which is $G$-equivariant and an isomorphism by construction. This finishes the proof of the exactness of the sequence~\eqref{eq:main}.

\medbreak

It is immediate that $\bbL$ restricts to an isomorphism on prime-to-$p$ torsion since $\Hom_{\gps}(P,\bbC^*)$ is $p^r$-torsion, where $|P|=p^r$, hence every $L\in \Tors{m}{\Pic^G(\QG)}$ with $m$ prime to~$p$ maps to zero under~$\Res^G_P$.

This finishes the proof of Theorem~\ref{thm:main}.
\end{proof}

\begin{Rem}
Construction~\ref{cons:u} describes the inverse of~$\bbL$ on prime-to-$p$ torsion.
\end{Rem}

Let us now connect these results over~$\bbC$ to positive characteristic objects. We recall some well-known facts, to facilitate cognition.

\begin{Rem}
\label{rem:bar-k}%
The group~$\TGP$ is always finite. (Indeed, every endotrivial module in~$\TGP$ is a direct summand of~$\kk(G/P)$ -- an explicit projector depending on~$u\in\AGP$ is given in~\cite{Balmer13b}. By Krull-Schmidt it follows that $\TGP$ has at most~$\dim_{\kk}(\kk(G/P))=[G:P]$ elements.) Also, the order of~$\TGP$ is prime to~$p$; see~\cite[Cor.\,5.3]{Balmer13b}. For an algebraic closure~$\bar\kk$ of~$\kk$, one can easily identify the image of $\TGP \hookrightarrow \mathrm{T}_{\bar\kk}(G,P)$; see~\cite[Cor.\,5.5]{Balmer13b}.

In fact, the group $\TGP$ ``stabilizes" once \emph{$\kk$ contains all roots of unity} by which we mean it contains all $m\ith$ roots of unity for all integers $m\geq 1$ prime to~$p$. Here, ``stabilization" means that $\mathrm{T}_{\kk}(G,P)\to \mathrm{T}_{\kk'}(G,P)$ is an isomorphism for every further extension $\kk\to \kk'$; see~\cite[Cor.\,5.5]{Balmer13b}. This condition is of course fulfilled if the field~$\kk=\bar \kk$ is algebraically closed, or simply if $\kk$ contains $\bar\bbF_p$, the algebraic closure of the prime field. Our Theorem~\ref{thm:main-intro} is another way of seeing why $\TGP$ stabilizes once $\kk$ contains all roots of unity, by giving it a topological interpretation:
\end{Rem}

\begin{Cor}
\label{cor:main}%
The prime-to-$p$ torsion $\Tors{p'}{\Pic^G(\QG)}$ is a finite subgroup of~$\Pic^G(\QG)$. For any field~$\kk$ of characteristic~$p$ which contains all roots of unity (see Remark~\ref{rem:bar-k}), we have an isomorphism as announced in Theorem~\ref{thm:main-intro}
\[
\TGP \simeq \ \Tors{p'}{\Pic^G(\QG)}
\]
where $\torsname_{p'}{}$ denotes the prime-to-$p$ torsion subgroup.
\end{Cor}

\begin{proof}
Let $\kk$ containing all roots of unity and let $e$ be the exponent of~$\TGP$. Let $m\geq 1$ be an integer, prime to~$p$ and divisible by~$e$.

By~\eqref{eq:AT}, the integer~$e$ is also the exponent of~$\AGP\simeq\TGP$ hence $u^m=1$ for all $u\in\AGP$. Thus every $u:G\to \kk^*$ in~$\AGP$ takes values in~$\mu_m(\kk)$. In other words, we can identify the group of \emph{$\kk$-valued} weak $P$-homomorphisms $\AGP$ with the set of functions $u:G\to \mu_m(\kk)$ satisfying~(WH\,\ref{it:WH1}-\ref{it:WH3}).

Consider now inside the group~ $\ACGP$ of \emph{complex-valued} weak $P$-homo\-mor\-phisms, the subgroup $\Tors{m}{\ACGP}$ of elements of order dividing~$m$. Again, this is just the subset of those functions $u:G\to \mu_m(\bbC)$ satisfying~(WH\,\ref{it:WH1}-\ref{it:WH3}).

\emph{Choose} now an isomorphism $\mu_m(\kk)\simeq\bbZ/m\simeq\mu_m(\bbC)$. This uses that $\kk$ contains all~$m\ith$ roots of unity. Combining the above we obtain an isomorphism
\begin{equation}
\label{eq:A-2-A}%
\AGP\simeq \mACGP.
\end{equation}
Since the left-hand side is independent of such~$m$ (prime to~$p$ and divisible by~$e$), we get $\Tors{p'}{\ACGP}=\Tors{e}{\ACGP}$. Using now Theorem~\ref{thm:main}, it follows that $\Tors{p'}{\Pic^G(\QG)}=\Tors{e}{\Pic^G(\QG)}\simeq\Tors{e}{\ACGP}$ via~$\bbL$. The latter is itself isomorphic to~$\AGP\simeq \TGP$ by a last instance of~\eqref{eq:A-2-A} and~\eqref{eq:AT}.
\end{proof}%

\begin{Rem}
The isomorphism of Corollary~\ref{cor:main} is essentially induced by the canonical homomorphism~$\bbL:\ACGP\to \Pic^G(\QG)$ of Section~\ref{sec:u-to-L}, up to the choice of an identification between $e\ith$ roots of unity in~$\kk$ and $e\ith$ roots of unity in~$\bbC$, for $e$ the exponent of~$\TGP$. Another choice of an isomorphism $\mu_e(\kk)\simeq\mu_e(\bbC)$ simply changes the isomorphism~\eqref{eq:A-2-A} by multiplication with some integer prime to~$e$, a rather harmless operation which is of course invertible.
\end{Rem}

Combining the above with Example~\ref{exa:trivial}, we obtain:
\begin{Cor}
The following properties of~$G$ and~$p$ are equivalent:
\begin{enumerate}[(i)]
\item
For $\kk=\bar\bbF_p$ the group~$\TGP$ consists only of one-dimensional representations $G\to \kk^*$.
\item[(i')]
For every field $\kk$ containing all roots of unity, the group~$\TGP$ consists only of one-dimensional representations $G\to \kk^*$.
\smallbreak
\item
Every $G$-equivariant complex line bundle on~$\QG$ which is torsion of order prime to $p$ is constant, \ie $\Tors{p'}{\Pic^G(*)}\to \Tors{p'}{\Pic^G(\QG)}$ is onto.
\qed
\end{enumerate}
\end{Cor}

\bigbreak
\textbf{Acknowledgments}\,: This paper is a satellite of a project with Serge Bouc, based on his insight on how ``sipp cohomology" in the sense of~\cite{Balmer15} might be connected to the Brown complex. I am extremely grateful to him for numerous detailed discussions, both at philosophical and technical levels. The idea for the present article originated in this interaction. I thank Burt Totaro for several useful discussions and for suggesting Remark~\ref{rem:Burt}. I also thank Ivo Dell'Ambrogio and Jacques Th\'evenaz for valuable comments. Finally, I thank Henning Krause and Bielefeld University for their hospitality during the preparation of this work.



\begin{thebibliography}{CMN14}

\bibitem[Bal13]{Balmer13b}
Paul Balmer.
\newblock Modular representations of finite groups with trivial restriction to
  {S}ylow subgroups.
\newblock {\em J. Eur. Math. Soc. (JEMS)}, 15(6):2061--2079, 2013.

\bibitem[Bal15]{Balmer15}
Paul Balmer.
\newblock Stacks of group representations.
\newblock {\em J. Eur. Math. Soc. (JEMS)}, 17(1):189--228, 2015.

\bibitem[Bou84]{Bouc84}
Serge Bouc.
\newblock Homologie de certains ensembles ordonn\'es.
\newblock {\em C. R. Acad. Sci. Paris S\'er. I Math.}, 299(2):49--52, 1984.

\bibitem[Bro75]{Brown75}
Kenneth~S. Brown.
\newblock Euler characteristics of groups: the {$p$}-fractional part.
\newblock {\em Invent. Math.}, 29(1):1--5, 1975.

\bibitem[CMN14]{CarlsonMazzaNakano14}
Jon~F. Carlson, Nadia Mazza, and Daniel~K. Nakano.
\newblock Endotrivial modules for the general linear group in a nondefining
  characteristic.
\newblock {\em Math. Z.}, 278(3-4):901--925, 2014.

\bibitem[CT04]{CarlsonThevenaz04}
Jon~F. Carlson and Jacques Th{\'e}venaz.
\newblock The classification of endo-trivial modules.
\newblock {\em Invent. Math.}, 158(2):389--411, 2004.

\bibitem[CT05]{CarlsonThevenaz05}
Jon~F. Carlson and Jacques Th{\'e}venaz.
\newblock The classification of torsion endo-trivial modules.
\newblock {\em Ann. of Math. (2)}, 162(2):823--883, 2005.

\bibitem[CT15]{CarlsonThevenaz15}
Jon~F. Carlson and Jacques Th{\'e}venaz.
\newblock The torsion group of endotrivial modules.
\newblock {\em Algebra Number Theory}, 9(3):749--765, 2015.

\bibitem[GGK02]{GuilleminGinzburgKarshon02}
Victor Guillemin, Viktor Ginzburg, and Yael Karshon.
\newblock {\em Moment maps, cobordisms, and {H}amiltonian group actions},
  volume~98 of {\em Mathematical Surveys and Monographs}.
\newblock American Mathematical Society, Providence, RI, 2002.
\newblock Appendix J by Maxim Braverman.

\bibitem[HY76]{HattoriYoshida76}
Akio Hattori and Tomoyoshi Yoshida.
\newblock Lifting compact group actions in fiber bundles.
\newblock {\em Japan. J. Math. (N.S.)}, 2(1):13--25, 1976.

\bibitem[KR89]{KnoerrRobinson89}
Reinhard Kn{\"o}rr and Geoffrey~R. Robinson.
\newblock Some remarks on a conjecture of {A}lperin.
\newblock {\em J. London Math. Soc. (2)}, 39(1):48--60, 1989.

\bibitem[Qui78]{Quillen78}
Daniel Quillen.
\newblock Homotopy properties of the poset of nontrivial {$p$}-subgroups of a
  group.
\newblock {\em Adv. in Math.}, 28(2):101--128, 1978.

\bibitem[Seg68]{Segal68}
Graeme Segal.
\newblock Equivariant {$K$}-theory.
\newblock {\em Inst. Hautes \'Etudes Sci. Publ. Math.}, (34):129--151, 1968.

\bibitem[Sym98]{Symonds98}
Peter Symonds.
\newblock The orbit space of the {$p$}-subgroup complex is contractible.
\newblock {\em Comment. Math. Helv.}, 73(3):400--405, 1998.

\bibitem[Th{\'e}93]{Thevenaz93}
Jacques Th{\'e}venaz.
\newblock Equivariant {$K$}-theory and {A}lperin's conjecture.
\newblock {\em J. Pure Appl. Algebra}, 85(2):185--202, 1993.

\bibitem[TW91]{ThevenazWebb91}
J.~Th{\'e}venaz and P.~J. Webb.
\newblock Homotopy equivalence of posets with a group action.
\newblock {\em J. Combin. Theory Ser. A}, 56(2):173--181, 1991.

\bibitem[Web87]{Webb87}
P.~J. Webb.
\newblock Subgroup complexes.
\newblock In {\em The {A}rcata {C}onference on {R}epresentations of {F}inite
  {G}roups ({A}rcata, {C}alif., 1986)}, volume~47 of {\em Proc. Sympos. Pure
  Math.}, pages 349--365. Amer. Math. Soc., Providence, RI, 1987.

\end{thebibliography}
\end{document}